\documentclass[11pt]{amsart}
\usepackage{amsmath, amssymb, amsthm, amsfonts,enumitem}
\usepackage[colorlinks = true, citecolor = blue]{hyperref}

\theoremstyle{plain}
\newtheorem{theorem}{Theorem}[section]

\newtheorem{lemma}[theorem]{Lemma}
\newtheorem{proposition}[theorem]{Proposition}
\newtheorem{fact}[theorem]{Fact}
\newtheorem*{theorem*}{Theorem}
\newcounter{MainCorollaryCounter}

\newcounter{MainTheoremCounter}
\newtheorem{MainTheorem}[MainTheoremCounter]{Theorem}

\theoremstyle{definition}
\newtheorem{definition}[theorem]{Definition}
\newtheorem*{setup}{Setup}
\newtheorem{example}[theorem]{Example}
\newcounter{MainConjectureCounter}
\newtheorem{MainConjecture}[MainConjectureCounter]{Conjecture}

\newcounter{MainPrimeConjectureCounter}
\newtheorem{MainPrimeConjecture}[MainPrimeConjectureCounter]{Conjecture}

\newtheorem{remark}[theorem]{Remark}

\newcommand{\textdef}[1]{\textit{#1}}

\newcommand{\barc}{\bar{c}}
\newcommand{\barh}{\bar{h}}
\newcommand{\barm}{\bar{m}}
\newcommand{\bary}{\bar{y}}
\newcommand{\bfone}{\mathbf{1}}
\newcommand{\conjunder}[1]{\sim_{#1}}

\DeclareMathOperator{\sym}{Sym} 
\DeclareMathOperator{\alt}{Alt} 
\DeclareMathOperator{\aut}{Aut}
\DeclareMathOperator{\inn}{Inn}

\DeclareMathOperator{\soc}{soc} 
\DeclareMathOperator{\wreath}{wr} 
\renewcommand{\wr}{\wreath}

\begin{document}
\title{A reduction theorem for primitive binary permutation groups}
\author{Joshua Wiscons}
\address{Department of Mathematics\\
Hamilton College\\
Clinton, NY 13323, USA}
\email{jwiscons@hamilton.edu}
\thanks{This material is based upon work supported by the National Science Foundation under grant no. OISE-1064446.}
\begin{abstract}
A permutation group $(X,G)$ is said to be binary, or of relational complexity $2$, if for all $n$, the orbits of $G$ (acting diagonally) on $X^2$ determine the orbits of $G$ on $X^n$ in the following sense: for all $\bar{x},\bar{y} \in X^n$,  $\bar{x}$ and $\bar{y}$ are $G$-conjugate if and only if every pair of entries from $\bar{x}$ is $G$-conjugate to the corresponding pair from $\bar{y}$.  Cherlin has conjectured that the only finite primitive binary permutation groups are $S_n$, groups of prime order, and affine orthogonal groups $V\rtimes O(V)$ where $V$ is a vector space equipped with an anisotropic quadratic form; recently he succeeded in establishing the conjecture for those groups with an abelian socle. In this note, we show that what remains of the conjecture reduces, via the O'Nan-Scott Theorem, to groups with a nonabelian simple socle. 
\end{abstract}
\maketitle

\section{Introduction}
This note focuses on binary permutation groups; that is, permutation groups of relational complexity $2$. The relational complexity of a group $G$ acting on a set $X$ is the smallest $k$ (if one exists) for which the orbits of $G$ on $X^k$ ``determine'' the orbits of $G$ on $X^n$ for all $n$. The precise meaning of ``determine'' will be made clear below, but whatever it is, one may rightly believe that being binary is a rather restrictive hypothesis. Here, we address the following conjecture of Cherlin, see \cite[\S3]{ChG00} and \cite[Conjecture~1]{ChG13}.

\begin{MainConjecture}\label{conj.A}
A finite \emph{primitive} binary permutation group is either $S_n$ acting naturally on $\{1,\ldots,n\}$, a cyclic group of prime order acting regularly, or an affine orthogonal group $V\rtimes O(V)$ where $V$ is a vector space equipped with an anisotropic quadratic form and $O(V)$ is the full orthogonal group.
\end{MainConjecture}

Relational complexity has its roots in Lachlan's classification theory for finite homogeneous structures (see \cite{LaA84} or \cite{KnLa87}); however, little was known about relational complexity  in specific cases until the work of Cherlin, Martin, and Saracino in \cite{CMS96}, which was followed-up by Saracino's remarkable and detailed analysis in  \cite{SaD99} and \cite{SaD00}. More examples were laid out in \cite{ChG00} where Conjecture~\ref{conj.A} also took form. The recent work \cite{ChG13} of Cherlin establishes the conjecture for groups with an abelian socle, i.e. \textdef{affine groups}, and here, using the O'Nan-Scott Theorem, we reduce what remains to the case of groups with a nonabelian simple socle, i.e. \textdef{almost simple groups}. Specifically, we prove the following.

\begin{MainTheorem}\label{thm.A}
If $G$ is a finite primitive binary permutation group, then either
\begin{enumerate}
\item $G$ is affine,
\item $G$ is almost simple, or
\item $G$ is a subgroup of a wreath product $H\wr S_m$ in its product action where $H$ is a primitive binary almost simple permutation group that is not $2$-transitive.
\end{enumerate}
\end{MainTheorem}

In light of Cherlin's solution for groups of affine type, Theorem~\ref{thm.A} reduces Conjecture~\ref{conj.A} to the following conjecture specific to the almost simple case. 

\begin{MainPrimeConjecture}\label{conj.Aprime}
The only finite primitive binary permutation group with a nonabelian simple socle is $S_n$ in its natural action on $\{1,\ldots,n\}$.
\end{MainPrimeConjecture}

It is worth emphasizing that one should be able to say something interesting under a more general $k$-ary hypothesis, but it is not clear, at present, what this might be. 
Our arguments do not seem to immediately generalize beyond the binary case, so we have avoided the general setting, for now.  

\section{Preliminaries}

We first fix some notation and conventions that will be used throughout the article. From this point on, \emph{all groups are finite}. A permutation group is just a group with a fixed, faithful action on some set.  To emphasize both the group and the set, we may call $(X,G)$ a permutation group to mean that $G$ is a group acting faithfully on the set $X$.  We will often consider the induced (diagonal) action of $G$ on a power $X^n$; the orbits of $G$ on $X^n$ will be called \textdef{$n$-types}. For $\bar{x},\bar{y} \in X^n$ and $H\le G$, we say that $\bar{x}$ and $\bar{y}$ are \textdef{$H$-conjugate}, denoted $\bar{x}\conjunder{H} \bar{y}, $ if they lie in the same $H$-orbit. We denote the pointwise stabilizer of a subset $Y\subseteq X$ by $G_Y$. Unless otherwise stated, groups will always act on the right.

\subsection{Primitive groups}

A permutation group $(X,G)$ is \textdef{primitive} if $X$ has no proper nontrivial $G$-invariant equivalence relations. This is equivalent to $G$ acting transitively with all point stabilizers being maximal proper subgroups of $G$.  Our main reference for primitive groups will be \cite{DiMo96}.

The analysis of primitive groups may be broken down according to the structure of the \textdef{socle} of the group, i.e. the subgroup generated by all minimal normal subgroups. We will denote the socle of $G$ by $\soc(G)$, and if $G$ is a primitive group, $\soc(G)$ will be a direct product of isomorphic simple groups. We now give a very rough statement of the O'Nan-Scott Theorem; one may see \cite[Section~4.8]{DiMo96} for more details. We will elaborate on the various types as they arise in our analysis.

\begin{fact}\label{fact.ONS}
Let $(X,G)$ be a finite primitive permutation group. Set $M:=\soc(G)$, and write $M=T^k$ for some simple group $T$. Then $(X,G)$ is of one of the following types.
\begin{description}
\item[Affine] $M$ is abelian and acts regularly.
\item[Regular nonabelian] $M$ is nonabelian and acts regularly.
\item[Almost simple] $M$ is nonabelian simple and does not act regularly.
\item[Diagonal] $X$ may be identified with $T^k$ modulo the equivalence relation given by the orbits of $T$ acting diagonally on $T^k$ by left multiplication (by the inverse); $M$ acts coordinatewise by right multiplication.
\item[Product]  $G$ is a subgroup of a wreath product $H\wr S_m$ in its product action where $H$ is primitive of almost simple or diagonal type and $M=\left(\soc(H)\right)^m$.
\end{description}
\end{fact}

\subsection{Relational complexity}
We will define the relational complexity of a permutation group below, but first, we mention the connection to homogeneous structures. Let $\mathbf{X} = (X; R_1, \ldots, R_m)$ be a relational structure; that is, for each $i$, there is a positive integer $n_i$ for which $R_i$ is an $n_i$-ary relation on $X$. Now, every $Y\subseteq X$ gives rise to an induced substructure $\mathbf{Y} = (Y; R^Y_1, \ldots, R^Y_m)$ where each $R^Y_i$ is the restriction of $R_i$ to $Y$, and the structure $\mathbf{X}$ is called \textdef{(ultra)homogeneous} if every isomorphism between induced finite substructures can be extended to an automorphism of $\mathbf{X}$. 

Consider the Petersen graph as a structure $(V; E)$ with $V$ being the set of vertices and $E$ the edge relation. It is not hard to see that $(V; E)$ is \emph{not} homogeneous since there exist independent triples of vertices that have a common neighbor as well as independent triples that do not have a common neighbor. However, if we define a ternary relation $N$ on $V$ such that $(v_1,v_2,v_3)\in N$ if and only if $v_1$, $v_2$, and $v_3$ have a common neighbor, then it can be shown that the expanded structure $(V; E, N)$ is indeed homogeneous.
Thus, although the Petersen graph is not homogeneous in the language of graphs, it becomes homogeneous after adding in an appropriate ternary relation. Moreover, the relation $N$ is definable (in a first-order way, without parameters) from $E$, so $(V; E, N)$ is model-theoretically equivalent to $(V; E)$.

Returning to the general setting, the \textdef{relational complexity} of $\mathbf{X}$ is defined to be the least $k$ (if one exists) such that $\mathbf{X}$ is equivalent to a homogeneous structure all of whose relations are $k$-ary.  Note that if $\mathbf{X}$ is \emph{finite} and if $\widetilde{\mathbf{X}}$ is the expansion of $\mathbf{X}$ obtained by adding in, for all $k\le |X|$, the $k$-ary relations corresponding to the orbits of $\aut(\mathbf{X})$ on $X^k$, then $\widetilde{\mathbf{X}}$ is homogeneous and equivalent to $\mathbf{X}$. In fact, $k\le |X|-1$ will do. As such, the relational complexity of a finite relational structure is always defined and is at most $|X|-1$.  The conclusion of the previous paragraph was that the relational complexity of the Petersen graph is $3$.

We now change our focus from $\mathbf{X}$ to $\aut(\mathbf{X})$ and translate the study of relational complexity into the language of permutation groups. We begin with a preliminary definition.

\begin{definition}
Let $(X,G)$ be a permutation group.
\begin{enumerate}
\item Tuples $\bar{x},\bar{y} \in X^n$ are  \textdef{equivalent} if they are $G$-conjugate, i.e. if they are in the same $n$-type.
\item Tuples $\bar{x},\bar{y} \in X^n$ are  \textdef{$k$-equivalent} ($k\le n$) if every $k$-subtuple from $\bar{x}$ is $G$-conjugate to the corresponding $k$-subtuple from $\bar{y}$.
\item We say that \textdef{$k$-types determine $n$-types} ($k\le n$) if for every $\bar{x},\bar{y} \in X^n$, $k$-equivalence of $\bar{x}$ and $\bar{y}$ implies equivalence of  $\bar{x}$ and $\bar{y}$. 
\end{enumerate}
\end{definition}

\begin{definition}
The \textdef{relational complexity} of a permutation group is the smallest $k$ for which $k$-types determine $n$-types for all $n\ge k$.
\end{definition}

\begin{example}
Every nontrivial permutation group has relational complexity at least $2$. Indeed, assume that $(X,G)$ is nontrivial, and let $x\neq y\in X$ be $G$-conjugate. Then the pairs $(x,x)$ and $(x,y)$ are $1$-equivalent but not $2$-equivalent.
\end{example}

\begin{remark}\label{rem.DistinctEntries}
With the previous example in mind, it is not hard to see that the relational complexity of a \emph{nontrivial} permutation group is the smallest $k\ge2$ for which $k$-types determine $n$-types \emph{of tuples with pairwise distinct entries} for all $n\ge k$. Indeed, if one assumes $k$-equivalence for $k\ge2$, then repeated entries in the first tuple must match up with conjugate repeated entries in the second. Certainly the repetitions may then be ignored. 
\end{remark}

\begin{example}\label{exam.Sn}
The natural action of $S_n$ on $\{1,\ldots,n\}$ is binary, i.e. of relational complexity $2$. In light of the previous remark, this is a simple consequence of $S_n$ being $n$-transitive. In fact, one finds that the only $k$-transitive permutation group with relational complexity less than $k+1$ is $S_n$ acting naturally on $\{1,\ldots,n\}$. As such, the natural action of $A_n$ on $\{1,\ldots,n\}$ has relational complexity $n-1$. 
\end{example}

See \cite{ChG13}, \cite{ChG00}, and  \cite{CMS96} for further background on relational complexity (also known as \textdef{arity}).

\section{Regular nonabelian type}
We now begin our proof of Theorem~\ref{thm.A} by moving through the various socle types laid out in Fact~\ref{fact.ONS}.
We start with the case of a regular nonabelian socle. Our goal is the following.

\begin{proposition}\label{prop.RegNonabelian}
Every finite primitive group of regular nonabelian type has relational complexity at least $3$.
\end{proposition}

We adopt the following setup for the remainder of the section; most of the necessary information for this type can be found in \cite[Theorem~4.7B]{DiMo96}.

\begin{setup}
Assume that $(X,G)$ is a finite primitive group with a regular, nonabelian socle $M$. Since $M$ acts regularly on $X$, we identify $X$ with $M$, and setting $H:= G_1$, we have $G=M\rtimes H$ with $M$ acting on $X$ by right translation and $H$ acting by conjugation. Furthermore, $M=T_1\times \cdots \times T_k$ with $k\ge 6$ and each $T_i$ isomorphic to a fixed nonabelian simple group $T$. (In this case, $G$ is isomorphic to a so-called twisted wreath product $T \operatorname{twr}_\varphi H$.)
\end{setup}

If $a \in X$, $m\in M$, and $h\in H$, we write the action of $mh$ on $a$ as $a\cdot mh:=(am)^h$ so as to avoid confusion with the product $amh$ in $G$. Note that $H$ acts on $X$ as  automorphisms, so in particular, $(a^{-1})\cdot h = (a\cdot h)^{-1}$. We first work for an analog of \cite[Corollary~1.4]{ChG13}; this will be Lemma~\ref{lem.RegNonabelian.Main}. 

\begin{lemma}\label{lem.RegNonabelian.TwoEquiv}
If $(a_1,a_2), (b_1,b_2)\in X^2$, then we have that $(a_1,a_2)\conjunder{G}(b_1,b_2)$ if and only if $a_1a_2^{-1}\conjunder{H}b_1b_2^{-1}$. 
\end{lemma}
\begin{proof}
If $(a_1,a_2)\cdot mh = (b_1,b_2)$ for some $m\in M$ and $h\in H$, then \[(a_1a_2^{-1})\cdot  h = [(a_1m)(a_2m)^{-1}]\cdot h = (a_1m)^h((a_2m)^{-1})^h = b_1 b_2^{-1}.\] Conversely, assume that $(a_1a_2^{-1})\cdot h = b_1b_2^{-1}$ for some $h\in H$. Then $g:=a_2^{-1}hb_2$ takes $a_1$ to $b_1$, and it is trivial to check that $a_2\cdot g = b_2$.
\end{proof}

\begin{lemma}\label{lem.RegNonabelian.Main}
Assume  $(X,G)$ is binary. If $a\in X$ and $h\in H$ are such that $[a,a\cdot h] = 1$, then there exists an $h'\in H$ swapping $a$ and $a\cdot h$.
\end{lemma}
\begin{proof}
The conclusion of the lemma is precisely that $(a,a\cdot h) \conjunder{H} (a\cdot h,a)$. Since $H$ acts on $X$ as automorphisms, this is equivalent to showing that $(a,a^{-1}\cdot h) \conjunder{H} (a \cdot h,a^{-1})$, which is, of course, equivalent to showing  that $(1,a,a^{-1}\cdot h) \conjunder{G} (1,a \cdot h,a^{-1})$.
Now, the action is binary, so it suffices to show that $(1,a,a^{-1}\cdot h) $ is $2$-equivalent to $(1,a \cdot h,a^{-1})$. Since $h$ and $h^{-1}$ fix $1$, the only nontrivial conjugacy to check is $(a,a^{-1}\cdot h) \conjunder{G} (a \cdot h,a^{-1})$, but this follows from the previous lemma and our assumption that $[a,a\cdot h] = 1$.
\end{proof}

\begin{definition}
A subset of $X$ (or tuple in $X^\ell$) is called \textdef{$H$-connected} if all elements of the subset (or entries of the tuple) are $H$-conjugate.
\end{definition}

By \cite[Theorem~4.7B(ii)]{DiMo96}, $H$ acts transitively on the set $\{T_1,\ldots,T_k\}$. Thus for any $a_1 \in T_1$ there exist $a_2,\ldots,a_k\in X$ such that each $a_i \in T_i$ and $\{a_1,\ldots,a_k\}$ is $H$-connected. Thus, we are guaranteed the existence of (many) $H$-connected subsets of $k$ \emph{commuting} elements of $X$, and every nontrivial  $a_1 \in T_1$ can be extended to at least one such subset. 

\begin{lemma}\label{lem.RegNonabelianConnected}
Assume  $(X,G)$ is binary. If $(a_1,\ldots,a_\ell)$ is an $H$-connected tuple of commuting elements of $X$, then $(a_1,\ldots,a_\ell)\conjunder{G}(a_1^{-1},\ldots,a_\ell^{-1})$. Further, if $a_1$ and $a_1^{-1}$ are $H$-conjugate, then $(a_1,\ldots,a_\ell)\conjunder{H}(a_1^{-1},\ldots,a_\ell^{-1})$.
\end{lemma}
\begin{proof}
Let $1\le i<j\le\ell$. Since $a_i$ and $a_j$ commute and are $H$-conjugate, we apply Lemma~\ref{lem.RegNonabelian.Main} to see that there is an element of $H$ swapping $a_i$ and $a_j$. Further, this implies that $a_ia_j^{-1}$ is $H$-conjugate to $a_ja_i^{-1}=a_i^{-1}a_j$, so we find that $(a_i,a_j)\conjunder{G}(a_i^{-1},a_j^{-1})$ by Lemma~\ref{lem.RegNonabelian.TwoEquiv}. Since the action is binary, we conclude that  $(a_1,\ldots,a_\ell)\conjunder{G}(a_1^{-1},\ldots,a_\ell^{-1})$. 

Now assume that $a_1$ is $H$-conjugate to $a_1^{-1}$. Since $(a_1,\ldots,a_\ell)$ is $H$-connected, we find that $a_i$ is $H$-conjugate to $a_i^{-1}$ for all $1\le i\le\ell$. Thus, for every $i$, we have that $(1,a_i) \conjunder{G} (1,a_i^{-1})$.  We have already seen that $(a_i,a_j)\conjunder{G}(a_i^{-1},a_j^{-1})$ for all  $1\le i<j\le\ell$, so the fact that $(X,G)$ is binary implies that $(1,a_1,\ldots,a_\ell)\conjunder{G}(1,a_1^{-1},\ldots,a_\ell^{-1})$. Of course, this is equivalent to $(a_1,\ldots,a_\ell)\conjunder{H}(a_1^{-1},\ldots,a_\ell^{-1})$.
\end{proof}

The final ingredient for our proof of Proposition~\ref{prop.RegNonabelian} is the following general (and likely well-known) lemma.

\begin{lemma}
If $T$ is a nonabelian simple group with an automorphism $\alpha$ of order $2$, then $\alpha$ inverts, but does not centralize, some element of $T$.
\end{lemma}
\begin{proof}
We work in $T\rtimes\langle \alpha \rangle$. Consider $[T,\alpha]$, as a set. It is easy to see that $\alpha$ inverts every element of $[T,\alpha]$. Now, towards a contradiction, assume that $[T,\alpha]\subseteq C(\alpha)$. If $t\in T$, then $\alpha$ centralizes $\alpha^t\alpha$, so $\alpha$ commutes with $\alpha^t$. We conclude that $\alpha^T \subseteq C(\alpha)$, and as the set $\alpha^T$ is $T$-normal, we find that $\alpha^T \subseteq C(\alpha^T)$. In particular, $A:=\langle \alpha^T\rangle$ is abelian and $T$-normal. Further, $A$ is certainly nontrivial, so as $T$ is nonabelian and simple, we see that $T\cap A =1$. We conclude that $A=\langle \alpha \rangle$, so $T$ centralizes $\alpha$. Since $\alpha$ is nontrivial, this is a contradiction.
\end{proof}

\begin{proof}[of Proposition~\ref{prop.RegNonabelian}]
Assume that the action is binary. We first claim that $N_{H}(T_1)/C_{H}(T_1)$ contains an involution. Now, \cite[Theorem~4.7B(iii)]{DiMo96} states that $N_{H}(T_1)$ has a composition factor isomorphic to $T_1$, but in fact, the proof shows that such a factor appears in $N_{H}(T_1)/C_{H}(T_1)$. By the Feit-Thompson Theorem, $N_{H}(T_1)/C_{H}(T_1)$ contains an involution.

We may now apply the previous lemma to see that there exists an $a_1\in T_1$ such that $a_1\neq a_1^{-1}$ and $a_1$ is $H$-conjugate to $a_1^{-1}$. Extend $a_1$ to an $H$-connected tuple $(a_1,a_2,\ldots,a_k)$ with each $a_i \in T_i$, and note that $a_i$ commutes with $a_j$ for all $i$ and $j$. By Lemma~\ref{lem.RegNonabelianConnected}, there is some $h\in H$ taking $(a_1,\ldots,a_k)$ to $(a_1^{-1},\ldots,a_k^{-1})$. Now, by \cite[Theorem~4.7B(ii)]{DiMo96}, $H$ acts \emph{faithfully} on $\{T_1,\ldots,T_k\}$. Since $h$ inverts some nontrivial element in each $T_i$, we see that $h$ must normalize each $T_i$. Thus, $h=1$, so $(a_1,\ldots,a_k)=(a_1^{-1},\ldots,a_k^{-1})$. We have a contradiction. 
\end{proof}

\section{Diagonal type}

We now move to groups of diagonal type. We aim to prove the following. 

\begin{proposition}\label{prop.Diagonal}
Every finite primitive group of diagonal type has relational complexity at least $3$.
\end{proposition}

Information on groups of diagonal type can be found in \cite[Section~4.5]{DiMo96}. We fix the following setup.

\begin{setup}
Assume that $(X,G)$ is a finite primitive group of diagonal type. For some integer $k\ge 2$, the socle of $G$ is $M:=T_1\times \cdots \times T_k$ where each $T_i$ is isomorphic to a fixed nonabelian simple group $T$. Fixing isomorphisms of each $T_i$ with $T$, we identify $X$ with the set $T_1\times \cdots \times T_k$ modulo the equivalence relation given by the orbits of $T$ acting diagonally on $X$ by left multiplication (by the inverse); $M$ acts coordinatewise by right multiplication. An arbitrary element of $X$ is written $[a_1,\ldots,a_k]$ with each $a_i\in T_i$ where $[a_1,\ldots,a_k] := \{(t^{-1}a_1,\ldots,t^{-1}a_k): t\in T\}$. Set $\bfone:=[1,\ldots,1]$ and $H:= G_\bfone$. Then $H$ acts on $X$ as a subgroup of $\aut T \times S_k$ where $\aut T$ acts diagonally and $S_k$ permutes the components. Further, $H$ contains $\inn T$.
\end{setup}

Our approach here will be similar to that for the regular nonabelian type in that we will again find an element of the point stabilizer $H$ that simultaneously ``inverts'' a large tuple from $X$. However, in this case, the entries of the ``inverted'' tuple will all come from $T_1$; where as, in the regular nonabelian case, different entries came from different $T_i$. As before, we first derive an analog of  \cite[Corollary~1.4]{ChG13}.

\begin{lemma}\label{lem.DiagLemma}
Assume  $(X,G)$ is binary. For any nontrivial $t\in T_1$ and $s\in \inn T$ there is an $h \in N_{H}(T_1)$ swapping $[t,1,\ldots,1]$ and $[t^s,1,\ldots,1]$.
\end{lemma}
\begin{proof}
We aim to show that \[(\bfone,[t,1,\ldots,1],[t^{s},1,\ldots,1]) \conjunder{G} (\bfone,[t^s,1,\ldots,1],[t,1,\ldots,1])\] 
as any element taking the first tuple to the second must necessarily lie in $H$, hence in $N_{H}(T_1)$ since $H \le \aut T \times S_k$.
Further, note that this equivalent to showing that \[(\bfone,[t,1,\ldots,1],[t^{-s},1,\ldots,1]) \conjunder{G} (\bfone,[t^s,1,\ldots,1],[t^{-1},1,\ldots,1]).\] 
Now, we are assuming that the action is binary, so as $H$ contains $\inn T$, it is enough to note that the map \[[a_1,a_2,\ldots,a_k] \mapsto [a_1t^s,a_2t,\ldots,a_kt]\] is an element of $G$, in fact $M$, that takes the pair $([t,1,\ldots,1],[t^{-s},1,\ldots,1])$ to $([t^s,1,\ldots,1],[t^{-1},1,\ldots,1])$.
\end{proof}

\begin{proof}[of Proposition~\ref{prop.Diagonal}]
Assume that $(X,G)$ is binary. Let $r\in T$ be the product of two noncommuting involutions from $T$. Then, $r$ is not an involution, and $r$ is $T$-conjugate to $r^{-1}$.  

    Let $r_1$ be the element of $T_1$ corresponding to $r$ under our fixed isomorphism. Enumerate the conjugacy class of $r_1$ in $T_1$ as $r_1,\ldots,r_n$, and set $x_i := [r_i,1,\ldots,1]$. By the previous lemma, there exists an $h_{i,j} \in N_{H}(T_1)$ swapping $x_i$ and $x_j$. Further, since each $r_i$ is $T_1$-conjugate to $r_i^{-1}$,  Lemma~\ref{lem.DiagLemma} also shows that there exists a $k_i \in N_{H}(T_1)$ swapping $x_i$ and $x_i^{-1}$, where $x_i^{-1}:=[r_i^{-1},1,\ldots,1]$. We now claim that there is an $h\in N_{H}(T_1)$ that simultaneously ``inverts'' every element of $\{x_1,\ldots,x_n\}$, i.e. that $(x_1,\ldots,x_n)\conjunder{N_{H}(T_1)} (x_1^{-1},\ldots,x_n^{-1}).$ As in the proof of Lemma~\ref{lem.DiagLemma}, this is equivalent to showing $(\bfone,x_1,\ldots,x_n) \conjunder{G} (\bfone,x_1^{-1},\ldots,x_n^{-1}).$  Since the action is assumed to be binary, it only remains to verify that $(x_i,x_j) \conjunder{G} (x_i^{-1},x_j^{-1})$, and setting $m:=(r_i^{-1},r_j,\ldots,r_j)\in M$, it is easily checked that $(x_i,x_j)h_{ij}m= (x_i^{-1},x_j^{-1})$. 

Let $\hat{h} \in \aut(T)$ be the automorphism of $T$ corresponding to $h$, and let $\mathcal{C}:=r^T$. Since $h$ simultaneously inverts the $x_i$, it is not hard to see that 
$\hat{h}$ must invert every element of $\mathcal{C}$, so when restricted to $\mathcal{C}$, $\hat{h}$ commutes with $\inn(T)$. Thus, $[\inn(T),\hat{h}]$ centralizes $\mathcal{C}$, but then $[\inn(T),\hat{h}]$ centralizes the subgroup generated by $\mathcal{C}$. As $T$ is simple, we find that $[\inn(T),\hat{h}]$ centralizes $T$, i.e. $\hat{h}$ is a central automorphism. One easily computes that for all $s,t \in T$ \[(s^{\hat{h}})^{t^{-1}t^{\hat{h}}} = ((tst^{-1})^{\hat{h}})^{t^{\hat{h}}} =(t^{-1})^{\hat{h}}(tst^{-1})^{\hat{h}}t^{\hat{h}}= (t^{-1}tst^{-1}t)^{\hat{h}} = s^{\hat{h}},\]
so $t^{-1}t^{\hat{h}}\in Z(T)$ for all $t\in T$. Since $T$ is nonabelian and simple, we conclude that $\hat{h}$ acts trivially on $T$, but as $\hat{h}$ does not centralize $r$, we have a contradiction. 
\end{proof}

\section{Product type}

Finally, we address product type.

\begin{proposition}\label{prop.Product}
If $G$ is a finite primitive binary group of product type then $G$ is a subgroup of $H\wr S_m$ in its product action where $H$ is a primitive binary almost simple permutation group that is not $2$-transitive.
\end{proposition}

Here, in addition to drawing from \cite{DiMo96}, we also utilize \cite{KoL89}. We fix the following setup.

\begin{setup}
Assume that $(X,G)$ is a finite primitive group of product type. Identify $X$ with $Y^k$ (with $k\ge 2$) and $G$ with a subgroup of $W:=H \wr S_k$ in its product action where $H$ is a primitive subgroup of $\sym(Y)$ of almost simple or diagonal type. Setting $N:=\soc(H)$, we have $M:=\soc(G)=N^k$. 

Now, let $\pi$ be the obvious projection from $W$ to $S_k$, and set $P$ to be the image of $G$ under $\pi$. Further, if $W_1$ is the preimage under $\pi$ of the stabilizer of the first coordinate, then $W_1$ factors as $W_1 = H \times (H^{k-1}\wr S_{k-1})$, and we let $\pi_1$ be the projection of $W_1$ onto the first factor $H$. Finally, set $G_1$ to be the image of $(W_1\cap G)$ under $\pi_1$.
\end{setup}

We first highlight two important properties of $P$ and $G_1$ (see the discussion following \cite[(2.3)]{KoL89}).

\begin{fact}\label{fact.PTrans}
The action of $P$ on $\{1,\ldots,k\}$ is transitive; the action of $G_1$ on $Y$ is primitive with $\soc(G_1) = N$. 
\end{fact}

We view the elements of $X=Y^k$ as row vectors and $m$-tuples of elements of $X$ as $m\times k$-matrices. The elements of the base group $H^k$ then act ``column-wise'' on the matrices, and the top group permutes the columns. The proof of Proposition~\ref{prop.Product} is essentially the following straightforward lemma modulo one outstanding case that we address below. 

\begin{lemma}[see {\cite[Theorem~1]{CMS96}}]
The relational complexity of $(X,G)$ is at least as big as the relational complexity of $(Y,G_1)$.
\end{lemma}
\begin{proof}
Let $r$ be the relational complexity of $(X,G)$; note that $r\ge 2$. We now consider $(Y,G_1)$ and show that $r$-types determine $m$-types for all $m\ge r$. Indeed, take two  $r$-equivalent tuples $\barc_1,\barc_2 \in Y^m$ viewed as $m\times 1$ matrices. Appealing to Remark~\ref{rem.DistinctEntries}, we also assume that neither $\barc_1$ nor $\barc_2$ have repeated entries. Fix $y\in Y$. Let $\barc$ be the the $m\times 1$ matrix with each entry equal to $y$, and form the $m\times k$ matrices $A=\left(\begin{smallmatrix}\barc_1 & \barc & \barc & \cdots & \barc\end{smallmatrix}\right)$ and $B=\left(\begin{smallmatrix}\barc_2 & \barc & \barc & \cdots & \barc\end{smallmatrix}\right)$. Now, viewing $A$ and $B$ as elements of $X^m$, we easily see that they are $r$-equivalent under the action of $G$ (in fact $W_1\cap G$). Thus, by assumption, there is some $g\in G$ taking $A$ to $B$. Since  the first columns of $A$ and $B$ are the only nonconstant columns, it must be that $g\in W_1\cap G$, so the image of $g$ in $G_1$ takes $\barc_1$ to $\barc_2$. Hence, the $r$-types of $(Y,G_1)$ determine $m$-types for all $m\ge r$, so the relational complexity of $(Y,G_1)$ is at most $r$.
\end{proof}

\begin{proof}[of Proposition~\ref{prop.Product}]
Let $(X,G)$ be as above, and now assume that the action is binary. Thus, by the previous lemma $(Y,G_1)$ is binary. Additionally, $\soc(G_1) = \soc(H) = N$, and since $H$ is a primitive group of almost simple or diagonal type, the same is true of $G_1$. By Proposition~\ref{prop.Diagonal}, the second option is not possible, so $G_1$ is almost simple. It remains to show that $(Y,G_1)$ is not $2$-transitive.

Assume $(Y,G_1)$ is $2$-transitive. Since $(Y,G_1)$ is binary, this implies that $G_1=\sym(Y)$. We now give an explicit example showing that this in turn implies that the $2$-types of $(X,G)$ do not determine the $4$-types. Thus, the action is not binary, which is a contradiction. 

Since $(X,G)$ is of product type, $\soc(G)$, hence $N=\soc(G_1)$, is nonabelian, so  $N = \alt(Y)$ with $|Y| \ge 5$. Thus, for every $\bary_1, \bary_2 \in Y^k = X$ and every $g=\barh\sigma \in G$ ($\barh\in H^k$ and $\sigma \in S_k$), there is a $g'\in G$ such that $(\bary_1, \bary_2) g' = (\bary_1, \bary_2)\sigma$. Indeed, since $G$ contains $N^k$, the $2$-transitivity of $N$ on $Y$ implies that there is an  $\barm \in \alt(Y)^k$ for which $(\bary_1, \bary_2)\barm \barh =  (\bary_1, \bary_2)$. Thus we say that \textdef{$G$ realizes $P$ on pairs from $X$}. 

We now identify $Y$ with $\{1,\ldots,\ell\}$ for some natural number $\ell \ge 5$. Consider the following matrices representing two tuples in $X^4$:
\[ A = \begin{pmatrix} 
1 & 1 & 1 & \cdots & 1 \\
1 & 2 & 2 & \cdots & 2 \\
3 & 3 & 3 & \cdots & 3 \\
3 & 4 & 4 & \cdots & 4
\end{pmatrix}
\text{ and }
B=\begin{pmatrix} 
1 & 1 & 1 & \cdots & 1 \\
1 & 2 & 2 & \cdots & 2 \\
3 & 3 & 3 & \cdots & 3 \\
4 & 3 & 4 & \cdots & 4
\end{pmatrix}\]
We first claim that they are $2$-equivalent, i.e. that any two rows from $A$ are $G$-conjugate to the corresponding two rows in $B$. Clearly we need only check pairs of rows that include the fourth. Since $(Y,N)$ is $2$-transitive, we easily see that first and fourth as well as second and fourth rows of $A$ are simultaneously conjugate (using elements in $N^k$) to the corresponding rows in $B$. Finally, we use our observation above that $G$ realizes $P$ on pairs from $X$ together with the fact that $P$ is transitive on the coordinates to see that $G$ contains an element (only permuting coordinates) that takes the third and fourth rows of $A$ to the corresponding rows of $B$. We conclude that $A$ and $B$ are $2$-equivalent, but they have no chance to be $4$-equivalent since the ``column patterns'' are different. For example, the first column of $A$ contains precisely two distinct entries, so the image of $A$ under an element of $G$ must also contain a column with precisely two distinct entries. 
\end{proof}

\section*{Acknowledgements} 
The author would like to acknowledge the warm hospitality of the Hausdorff Research Institute for Mathematics where the work in this article began in the fall of 2013 during the trimester program on Universality and Homogeneity. The author is also grateful to Gregory Cherlin for the several enlightening and enjoyable conversations about relational complexity while in Bonn. Additionally, the author is thankful to the anonymous referee for a very careful reading of the paper and many helpful suggestions.

\bibliographystyle{alpha}
\bibliography{WisconsBib}
\end{document}